\documentclass{article}

\usepackage{amssymb, amsmath, amsthm, verbatim, amstext, amsfonts, enumerate}

\usepackage[latin1]{inputenc}

\newcommand{\set}[1]{\ensuremath{\{#1\}}}
\newcommand{\abs}[1]{\ensuremath{|#1|}}

\newcommand{\diam}{\textnormal{diam}}

\newcommand{\inv}{\ensuremath{^{-1}}}

\newcommand{\Aut}{\textnormal{Aut}}

\newcommand{\sm}{\ensuremath{\setminus}}
\newcommand{\es}{\ensuremath{\emptyset}}

\newcommand{\sub}{\subseteq}

\newcommand{\seq}[3]{(#1_#2)_{#2\in #3}}
\newcommand{\sequ}[1]{(#1_i)_{i\in{\mathbb N}}}

\theoremstyle{definition}
\newtheorem{Def}{Definition}[section]
\newtheorem{Exam}[Def]{Example}

\theoremstyle{plain}
\newtheorem{Lem}[Def]{Lemma}
\newtheorem{Cor}[Def]{Corollary}

\newtheorem{Tm}[Def]{Theorem}

\newtheorem{Claim}[Def]{Claim}

\newcommand{\nat}{{\mathbb N}}

\newcommand{\AF}{\ensuremath{\mathcal A}}
\newcommand{\BF}{\ensuremath{\mathcal B}}
\newcommand{\CF}{\ensuremath{\mathcal C}}

\newcommand{\GF}{\ensuremath{\mathcal G}}

\newcommand{\SF}{\ensuremath{\mathcal S}}
\newcommand{\TF}{\ensuremath{\mathcal T}}

\newcommand{\WF}{\ensuremath{\mathcal W}}
\newcommand{\XF}{\ensuremath{\mathcal X}}
\newcommand{\YF}{\ensuremath{\mathcal Y}}

\begin{document}

\title{End-transitive graphs}
\author{Matthias Hamann\\ Universit\"at Hamburg}
\date{}
\maketitle
\begin{abstract}
We investigate the structure of connected graphs, not necessarily locally finite, with infinitely many ends. On the one hand we study end-transitive such graphs and on the other hand we study such graphs with the property that the stabilizer of some end acts transitively on the vertices of the graph.
In both cases we show that the graphs have a tree-like structure.
\end{abstract}

\section{Introduction}

Woess~\cite{W} asked for a classification of the locally finite connected graphs with infinitely many ends and with an end-transitive automorphism group.
M\"oller~\cite{M} and Nevo~\cite{N} independently described these graphs.
The essence of their work is that they are similar to semi-regular trees, in particular, that they are quasi-isometric to semi-regular trees.
(A tree is {\em semi-regular} if all the vertices in each set of its natural bipartition have the same degree.)
The first theorem we shall prove is a generalization of their results to graphs that are not necessarily locally finite.
Before we state the theorem let us briefly define an abbreviation:
For a graph $G$, a vertex $x\in VG$, and $R\in\nat$ let $G(x,R)$ denote the union of the balls $B_R(x^\varphi)$, where $\varphi$ ranges over all automorphisms of~$G$.

\begin{Tm}\label{MTm10Short}
Let $G$ be a connected graph with infinitely many ends such that $G$ is end-transitive and $\Aut(G)$ fixes no vertex set of finite diameter.
For every $x\in VG$ there is an $R\in\nat$ such that $G(x,R)$ is quasi-isometric to a tree and $G-G(x,R)$ does not contain a ray.
\end{Tm}

The assumptions of Theorem~\ref{MTm10Short} (infinitely many ends, end-transitivity, no vertex set of finite diameter fixed by the automorphism group) are necessary: whenever we omit one of them, the conclusion of Theorem~\ref{MTm10Short} fails.

\bigskip

A second problem of Woess~\cite{W} is whether there is a classification of the locally finite connected graphs with infinitely many ends such that the stabilizer of one end acts transitively on the vertices of the graph.
He conjectured that such graphs are quasi-isometric to trees, which was subsequently proved by M\"oller~\cite{M2}.
This was generalized by Kr\"on~\cite{K2} to graphs of arbitrary cardinality with infinitely many {\em edge ends} (these are equivalence classes of rays, where two rays are equivalent if no finite set of edges separates them).
We prove here the corresponding result for (vertex) ends:

\begin{Tm}\label{mainTmPart2OfMb}
Let $G$ be a connected graph with infinitely many ends and with automorphism group $\Gamma$ such that for some end $\omega$ of~$G$ its stabilizer $\Gamma_\omega$ acts transitively on the vertices of $G$.
Then $G$ is quasi-isometric to a semi-regular tree with minimum degree $2$.
\end{Tm}

A further result of M\"oller~\cite{M2} is that in graphs such that the stabilizer of an end acts transitively on the vertices this stabilizer also acts transitively on the other ends.
In graphs that are not necessarily locally finite this is not the case.
But every orbit of the ends (other than the fixed end) is dense in the end space:

\begin{Tm}\label{mainTmPart2OfM}
Let $G$ be a connected graph with infinitely many ends and with automorphism group $\Gamma$. If $\Gamma_\omega$ for some end $\omega$ acts transitively on $VG$, then for any end $\omega'\ne\omega$ the $\Gamma_\omega$-orbit of the ends of~$G$ that contains $\omega'$ is dense in the set of all ends of~$G$.
\end{Tm}

\section{Preliminaries}

\subsection{Ends of graphs}\label{Ends}

Throughout this paper we use the terms and notation from \cite{D} if not stated otherwise.
In particular, a {\em ray} is a one-way infinite path. Two rays in a graph $G$ are {\em equivalent} if there is no finite vertex set $S$ in $G$ such that the two rays lie eventually in distinct components of $G-S$. The equivalence of rays is an equivalence relation whose classes are the {\em (vertex) ends} of $G$.
If we just talk of ends of graphs in this paper we always think of vertex ends.
All other end types - which we will define in a moment - will be stated concretely.

By replacing the finite vertex set $S$ in the definition of ends by a finite edge set one obtains {\em edge ends}.
Obviously for every graph there is a canonical map from its ends to its edge ends which is surjective but in general not injective.

A {\em metric ray} is a ray such that no infinite subset of its vertices has finite diameter.
Two metric rays are {\em metrically equivalent} if for every vertex set $S$ of finite diameter both rays lie eventually in the same component of $G-S$.
The equivalence classes of metrically equivalent metric rays are the {\em metric ends} of a graph.
Just as the ends are a refinement of edge ends, the metric ends are a refinement of ends.
A group acts {\em metrically almost transitively} on a graph $G$ if there is an $r\in\nat$ such that for every $x\in VG$ there is $G(x,r)=G$.
See \cite{K,KM} for more details on metric ends and metrically almost transitive graphs.

An end is {\em global} if every ray in that end is a metric ray.
If conversely there is no metric ray in an end this is a {\em local} end.
If an end not local, then it is a {\em non-local} end.
So an end is non-local if it contains a metric ray.

A vertex $x\in VG$ {\em dominates} an end $\omega$ if there is a ray $R$ in~$\omega$ and an infinite set of (except for $x$) pairwise disjoint $x$-$R$-paths.
An end $\omega$ is {\em thin} if there is an $n\in\nat$ such that there are at most $n$ disjoint rays in~$\omega$.
If the automorphism group of a graph acts transitively on the ends of that graph then the graph is {\em end-transitive}.

An automorphism $\alpha$ of a graph $G$ is a {\em translation} if there is no finite vertex set fixed by $\alpha$.

Let $X$ and $Y$ be metric spaces.
A map $\varphi:X\to Y$ is a {\em quasi-isometry} from $X$ to~$Y$ if there are constants $C\geq 1$, $D\geq 0$ such that for all $x,z\in X$ there is
$$\frac{1}{C}d(x^\varphi,z^\varphi)-D\leq d(x,z)\leq Cd(x^\varphi,z^\varphi)+D$$
and for all $y\in Y$ there is
$$d(y,X^\varphi)\leq D.$$

\subsection{Structure trees}\label{StructureTree}

Let $G$ be a connected graph and let $A,B\sub V(G)$ be two vertex sets. The pair $(A,B)$ is a \emph{separation} of $G$ if $A\cup B= V(G)$ and $E(G[A])\cup E(G[B])=E(G)$.
The \emph{order} of a separation $(A,B)$ is the order of its {\em separator} $A\cap B$ and the subgraphs $G[A\sm B]$ and $G[B\sm A]$ are the \emph{wings} of $(A,B)$.
With $(A,\sim)$ we refer to the separation $(A,(V(G)\sm A) \cup N(V(G)\sm A))$.
A {\em cut} is a separation $(A,B)$ of finite order with non-empty wings such that the wing $G[A\sm B]$ is connected and such that no proper subset of~$A\cap B$ separates the wings of~$(A,B)$.
A {\em cut system} of~$G$ is a non-empty set $\SF$ of separations $(A,B)$ of~$G$ satisfying the following three properties.
\begin{enumerate}[1.]
\item If $(A,B)\in\SF$ then there is an $(X,Y)\in \SF$ with $X\sub B$.
\item Let $(A,B)\in\SF$ and $C$ be a component of $G[B\sm A]$. If there is a separation $(X,Y)\in\SF$ with $X\sm Y\sub C$, then the separation $(C\cup N(C),\sim)$ is also in $\SF$.
\item If $(A,B)\in\SF$ with wings $X,Y$ and $(A',B')\in\SF$ with wings $X',Y'$ then there are components $C$ in $X\cap X'$ and $D$ in $Y\cap Y'$ or components $C$ in $Y\cap X'$ and $D$ in $X\cap Y'$ such that both $C$ and $D$ are wings of separations in~$\SF$.
\end{enumerate}

Two separations $(A_0,A_1),(B_0,B_1)\in \SF$ are \emph{nested} if there are $i,j\in \set{0,1}$ such that one wing of $(A_i\cap B_j,\sim)$ does not contain any component $C$ with ${(C \cup N(C),\sim)\in\SF}$ and $A_{1-i}\cap B_{1-j}$ contains $(A_0\cap A_1)\cup(B_0\cap B_1)$.
A cut system is {\em nested} if each two of its cuts are nested.

A cut in a cut system $\SF$ is {\em minimal} if its order in $\SF$ is minimal.
A {\em minimal cut system} is a cut system all whose cuts are minimal and thus have the same order.

Let us describe two minimal cut systems one of which was introduced by Dunwoody and Kr\"on~\cite[Example 2.2]{DK}.
Both will be used in the proofs of the main results.

\begin{Exam}\label{Example}
Let $G$ be a connected infinite graph with at least two ends (two non-local ends).
Let $n$ be the smallest order of a finite vertex set $X$ such that there are at least two components in $G-X$ that contain a ray (a metric ray) each.
Let $\SF$ be the set of all cuts $(A,B)$ with order $n$ such that both $G[A]$ and $G[B]$ contain a ray (a metric ray).
Then $\SF$ is a minimal cut system.
\end{Exam}

An {\em $\SF$-separator} is a vertex set $S$ that is a separator of some separation in $\SF$.
Let $\WF$ be the set of $\SF$-separators.
An \emph{$\SF$-block} is a maximal induced subgraph $X$ of~$G$ such that
\begin{enumerate}[(i)]
\item for every $(A,B)\in\SF$ there is $V(X)\sub A$ or $V(X)\sub B$ but not both;
\item there is some $(A,B)\in\SF$ with $V(X)\subseteq A$ and $A\cap B\subseteq V(X)$.
\end{enumerate}
Let $\BF$ be the set of $\SF$-blocks.
For a nested minimal cut system $\SF$ let $\TF$ be the graph with vertex set $\WF\cup\BF$. Two vertices $X,Y$ of~$\TF$ are adjacent if and only if either $X\in\WF$, $Y\in\BF$, and $X\sub Y$ or $X\in\BF$, $Y\in\WF$, and $Y\sub X$.
Then $\TF=\TF(\SF)$ is called the \emph{structure tree} of~$G$ and $\SF$ and by Lemma~6.2 of~\cite{DK} it is indeed a tree.

An {\em $\SF$-slice} is the induced subgraph $G[Z]$ of a component $Z$ of $G-(A\cap B)$ with $(A,B)\in\SF$ such that $(Z\cup (A\cap B),\sim)\notin\SF$.

A separation $(A,B)\in\SF$ {\em separates} two $\SF$-blocks, two ends of~$G$, or an $\SF$-block and an end of~$G$ if the blocks intersects non-trivially with distinct wings of $(A,B)$, if each two rays of distinct ends lie eventually in distinct wings of $(A,B)$, or if each ray eventually lies in that $(A,B)$-wing that intersects with the block trivially.
If one separator $S$ separates two vertices of another separator $S'$ the separators $S$ and $S'$ {\em crosses}.
It is a consequence of Lemma~3.3 and Theorem~3.5 of~\cite{DK} that two minimal separations are nested if and only if their corresponding separators do not cross.

A cut system $\SF$ of a connected graph $G$ is {\em basic} if $\SF$ is minimal nested and an $\Aut(G)$-invariant cut system such that $\SF$ is a subsystem of the minimal cut system given in Example~\ref{Example} and if all separators $A\cap B$ with $(A,B)\in \SF$ belong to the same $\Aut(G)$-orbit.

We state here that part of Theorem~7.2 of~\cite{DK} that we shall use here.

\begin{Tm}\label{basicTreeExists}\label{VertexCuts}
For every graph $G$ with at least two ends (two non-local ends) there is a basic cut system $\SF$ of~$G$.\qed
\end{Tm}

A ray $R$ {\em corresponds} to a vertex $X$ of~$\TF$ if $X$ is a block and $R\cap X$ is infinite.
A ray $R$ {\em corresponds} to an end $\omega$ of~$\TF$ if for any ray $P$ in~$\omega$ and for every $\SF$-separator $S$ on~$P$ all but finitely many vertices of $R$ lie in the same component of $G-S$ as that $\SF$-block which is in~$\TF$ adjacent to~$S$ and which separates $S$ from $\omega$ in~$G$.
Obviously a ray of~$G$ corresponds either to a vertex of~$\TF$ or to an end of~$\TF$.
As all rays in the same end have to correspond to the same vertex or end of~$\TF$, we also say that the end {\em corresponds} to that end or vertex of~$\TF$.

For a cut $(A,B)$ and a minimal cut system $\SF$ let $m_\SF(A,B)$ denote the number of distinct $\SF$-separators $S$ such that there is one $\SF$-separation that is not nested with $(A,B)$ and that has $S$ as its separation.
By \cite[Theorem~3.5, Lemma~4.1]{DK} the value $m_\SF(A,B)$ is finite.

\begin{Lem}\label{thirdCutSystem}
Let $G$ be a connected graph.
Let $\CF$ be a minimal cut system.
Let $\SF_1$ and $\SF_2$ be two nested subsystems of~$\CF$ each with a basic structure tree. Suppose that there are separations of $\SF_1$ and $\SF_2$ that are not nested.
Then there is a nested subsystem $\SF$ of~$\CF$ with a basic structure tree such that $\SF\cup \SF_2$ is a nested cut system and $m_{\SF_1}(A,B)<m_{\SF_1}(A',B')$ for all $(A,B)\in\SF, (A',B')\in\SF_2$.
\end{Lem}

\begin{proof}
Let $(A_1,B_1)\in\SF_1$ such that $(A_1,B_1)$ is not nested with all $(A,B)\in\SF_2$.
We choose $(A_2,B_2)\in\SF_2$ such that the intersection $X$ of one wing of $(A_2,B_2)$ with $A_1\cap B_1$ is minimal but not empty and such that the $\SF_2$-block containing $X$ is in the structure tree $\TF_2$ adjacent to $A_2\cap B_2$.
We may assume that $X\sub A_2$.
Then there is a component $C$ of $G-(A_1\cap B_1)-(A_2\cap B_2)$ such that $X\sub NC$ and $(C\cup NC,\sim)$ is a minimal cut.
Let $\SF$ be the set of all those cuts such that their separator is $NC^\alpha$ for any $\alpha\in\Aut(G)$.
We just have to prove that $\SF$ fulfills the claims of the lemma, so we have to prove that $\SF$ is a nested cut system, that $\SF_2\cup \SF$ is nested and that $m_{\SF_1}(A,B)<m_{\SF_1}(A',B')$ for all $(A,B)\in\SF, (A',B')\in\SF_2$.

By the minimal choice of~$X$ it follows that $\SF_2\cup \SF$ is nested.
So it remains to prove the inequality and that $\SF$ is nested.
Let us first prove the inequality.
Since each $\SF_1$-separation which is nested with $(A_2,B_2)$ also has to be nested with $(C\cup NC,\sim)$ by the minimal choice of~$X$, the inequality holds with $\leq$ instead of $<$, namely $m_{\SF_1}(C\cup NC,\sim)\le m_{\SF_1}(A_2,B_2)$.
But on the other hand there is the $\SF_1$-separation $(A_1,B_1)$ that is nested with $(C\cup NC,\sim)$ but not nested with $(A_2,B_2)$ and hence the inequality is strict.
Let us finally show that $\SF$ is nested.
Let $S:=A\cap B$ and let $\alpha\in\Aut(G)$ with $S^\alpha$ in the same component of $G-S$ in which $X$ lies.
By the choice of~$S$ and $X$ we know that $S^\alpha$ does not cross $A_1\cap B_1$.
Thus there is a component $D$ of $G-NC$ such that $S^\alpha\sub D\cup ND$.
The separator $(A_1\cap B_1)^\alpha$ crosses with $S^\alpha$ and thus it has to lie in the same component of $G-(A\cap B)$ as $S^\alpha$ does.
By a similar argument as before we know that $S$ does not separate $X^\alpha$ from $S^\alpha$ and thus both $S^\alpha$ and $X^\alpha$ do not intersect with the component of $G-(C\cup NC)$ that intersects with $X$ non-trivially.
Thus there are two $\SF$-separation with corresponding separators $NC$ and $NC^\alpha$ that are nested and as mentioned before this implies by arguments of~\cite{DK} that $NC$ and $NC^\alpha$ do not cross.
Thus $\SF$ is a nested cut system.
\end{proof}

\section{Structure trees and semi-regular trees}

In this section we prove that every basic structure tree of any connected graph whose automorphism group acts transitively on the non-local ends and fixes no vertex set of finite diameter is a semi-regular tree or a subdivided semi-regular tree.

Although the proof of the following lemma is similar to arguments in~\cite{N,SW} we proof it here because of the last part claimed.

\begin{Lem}\label{TranslExists}\label{TranslationExists}
Let $T$ be a tree and $\Gamma\leq \Aut(T)$ such that $\Gamma$ acts transitively on one set $A$ of the natural bipartition $A\cup B$ of~$VT$. 
Then for every path $x_0\ldots x_4$ of length $4$ between two vertices of~$A$ there is an automorphism $g\in\Gamma$ such that $g$ is a translation on~$T$ and either $x_0^g=x_2$ or $x_0^g=x_4$.
\end{Lem}

\begin{proof}
There is an automorphism $g\in\Gamma$ with $x_0^g=x_2$.
If $x_1^g\ne x_1$ then $g$ is a translation as claimed.
So let us assume that $x_1^g= x_1$.
There is an automorphism $h\in\Gamma$ with $x_0^h=x_4$.
If $x_2^h\ne x_2$ then $h$ is a translation as claimed.
Thus let us assume that $x_2^h= x_2$.
Let $f:=gh$.
Then $x_0^f=x_2$ and $x_1^f\ne x_1$. Hence $f$ is a translation and the lemma is proved.
\end{proof}

\begin{Lem}\label{EndsOfTAreNonLocal}
Let $G$ be a connected graph with infinitely many non-local ends such that $\Gamma:=\Aut(G)$ acts transitively on the non-local ends of~$G$.
Let $\SF$ be a basic cut system such that each $\SF$-separation separates metric rays.
If $\Gamma$ fixes no vertex set of finite diameter, then no end of~$\TF$, the structure tree of~$G$ and $\SF$, corresponds to a local end of~$G$.
\end{Lem}

\begin{proof}
We first remark that for every $n\in\nat$ there is a pair of $\SF$-separators with distance at least $n$ as otherwise the union of all $\SF$-separators is a vertex set of finite diameter.
Let us suppose that there is an end of~$\TF$ that corresponds to a local end of~$G$.
Then there is a ray in~$\TF$ and a vertex $x$ of~$G$ such that $x$ lies in all the vertices of that ray as otherwise there are infinitely many disjoint $\SF$-separators on that ray and thus the end of~$\TF$ corresponds to a non-local end of~$G$.
Similar to Lemma~\ref{TranslExists} there is an $\alpha\in\Gamma_x$, in the stabilizer of~$x$ in~$\Gamma$, that acts on~$\TF$ like a translation and thus $x$ lies in all vertices of the uniquely determined $\alpha$-invariant double ray $R$.
If $\TF$ has just two ends, then all separators lie on~$R$ and thus the intersection of all the separators is non-empty, of finite diameter, and $\Gamma$-invariant, but no such vertex set exists by the assumptions.
Hence we know that $\TF$ has infinitely many ends.

Let $S_0,S_1,S_2$ be three distinct $\SF$-separators such that $S_0$ and $S_1$ lie on~$R$, such that there is an $n\in\nat$ with $S_0^{\alpha^n}=S_1$, such that $S_0$ and $S_2$ are disjoint and also $S_1$ and $S_2$ are disjoint, and such that the shortest path from $S_2$ to~$R$ meets $R$ in the intervall from $S_0$ to~$S_1$.
Let $g\in\Gamma$ with $S_0^g=S_2$.
If $g$ does not act like a translation on~$\TF$, then $\alpha^{-n}g$ acts like a translation on~$\TF$ and so in each case there is an automorphism $h$ of~$G$ that acts on~$\TF$ like a translation.
The ends of~$\TF$ defined by the unique $h$-invariant double ray $P$ has to be non-local ends of~$G$ since there is an infinite pairwise disjoint subsequence, namely $(S_2^{h^i})_{i\in\nat}$, of the $\SF$-separators on~$P$ for each subray.
Let $S_i$ be on~$P$ with $i\in\{0,1\}$.
Then there is an $m\le d(S_i,S_2)$ such that each vertex of~$G$ lies in at most $m$ vertices of~$P$.

We can construct a ray $Q$ in~$\TF$ with all double rays $P^\gamma$ and $R^\gamma$ with $\gamma\in\Gamma$ such that there are infinitely many subpaths of~$Q$ of length $m+1$ whose intersection is non-trivial but such that $Q$ defines a non-local end of~$G$:
This ray just have to contain at least $m+1$ vertices from $P^\gamma$, then continue on some $R^{\alpha^k\gamma}$ until there is some vertex disjoint to all the vertices on $P^\gamma$ and then again continuing on some $P^{\gamma'}$ for at least $m+1$ vertices.
By repeating this process infinitely many times we finally have constructed a ray in~$\TF$ that has to correspond to some non-local end of~$G$.
But this leads to a contradiction to the transitivity of~$\Gamma$ on the non-local ends of~$G$ and hence no local end of~$G$ can correspond to an end of~$\TF$.
\end{proof}

By changing the cut system to a cut system $\SF$ such that every wing of every separation just contains a ray we obtain the following corollary of the proof of Lemma~\ref{EndsOfTAreNonLocal}.

\begin{Cor}\label{LocEndTrans=>NoEndInTree}
Let $G$ be a connected graph with infinitely many local ends such that $\Gamma:=\Aut(G)$ acts transitively on the ends of~$G$.
Let $\SF$ be a basic cut system such that each $\SF$-separation separates rays.
Then either $\Gamma$ fixes no vertex set of finite diameter or $\TF$, the structure tree of~$G$ and $\SF$, does not contain any ray.\qed
\end{Cor}

By the next lemma we show that every structure tree of a basic cut system that contains a ray is essentially a semi-regular tree.

\begin{Lem}\label{nearlySemi-regular}
Let $G$ be a connected graph with infinitely many non-local ends such that $\Gamma=\Aut(G)$ acts transitively on the non-local ends of~$G$ and fixes no vertex set of finite diameter.
Let $\SF$ be a basic cut system of~$G$ such that each $\SF$-separation separates non-local ends and let $\TF$ be the structure tre of~$G$ and $\SF$.
If $\TF$ contains some ray, then the set of $\SF$-blocks consists of at most two $\Gamma$-orbits.

In particular then there are two different cases: either $\Gamma$ has precisely two orbits on $V\TF$, or the separator vertices in $\TF$ have degree $2$ and there are precisely three $\Gamma$-orbits on~$V\TF$.
\end{Lem}

\begin{proof}
Let us first suppose that every $\SF$-separator lies in at most $2$ $\SF$-blocks.
Then there are at most two $\Gamma$-orbits on $\SF$.
Thus there are at most two $\Gamma$-orbits on the set of $\SF$-blocks.

Let us now suppose that every $\SF$-separator lies in at least $3$ distinct $\SF$-blocks.
If there are at least two $\Gamma$-orbits on the set of $\SF$-blocks, we can construct two rays $R$ and $P$ such that the ends $\omega_R$ and $\omega_P$ defined by $R$ and $P$, respectively, are not in the same $\Gamma$-orbit:
There is a ray $R$ such that every fourth vertex lies in some $\Gamma$-orbit $\XF$ of $\SF$-blocks which is avoided completely by a second ray $P$.
As the ends $\omega_R$ and $\omega_P$ of $\TF$ corresponds uniquely to some non-local ends $\widehat{\omega}_R$ and $\widehat{\omega}_P$ of~$G$ by Lemma~\ref{EndsOfTAreNonLocal} and $\Gamma$ acts transitively on the non-local ends of~$G$, there is some $\alpha$ with $\widehat{\omega}_R^\alpha=\widehat{\omega}_P$ and thus $\omega_R^\alpha=\omega_P$.
As $\TF$ is a tree, there has to be some vertex in $\XF$ on $P$, in particular every fourth vertex of~$P$ must be an element of~$\XF$.
Since this is not the case, we get a contradiction.
\end{proof}

\begin{Lem}\label{TFHasEnd}
Let $G$ be a connected graph with infinitely many non-local ends such that $\Gamma=\Aut(G)$ acts transitively on the non-local ends of $G$.
If there is no vertex set of finite diameter invariant under $\Gamma$, then the structure tree $\TF$ of~$G$ and any basic cut system $\SF$ has infinitely many ends.
\end{Lem}

\begin{proof}
We just have to prove that $\TF$ has some end $\omega$.
If this is the case, then we know that the non-local end $\widehat{\omega}$ of~$G$ corresponding to $\omega$ has infinitely many images under $\Gamma$ and thus also $\omega$ must have infinitely many images under $\Gamma$ as any end of~$\TF$ corresponds to precisely one non-local end of~$G$.

So let us suppose that $\TF$ has no end.
As $\Gamma$ acts transitively on the separator vertices of~$\TF$, the diameter of $\TF$ is at most $4$.
If the diameter is $2$, then there is a unique $\SF$-separator $S$ in~$G$ and thus $S$ is a vertex set of finite diameter invariant under $\Gamma$ in contradiction to our assumption.
Hence we know that $\TF$ has diameter $4$.
Our aim is to show that also in this case the vertex set of all those vertices that lie in any $\SF$-separator is a vertex set of finite diameter.

Let $X$ be that $\SF$-block that is in~$\TF$ adjacent to all separator vertices.
We will prove that $X$ contains some non-local end of~$G$.
As $G$ contains infinitely many non-local ends and $\Gamma$ acts transitively on those ends, all non-local ends must lie in~$X$ as $X$ is uniquely determined in~$\TF$.
But as the separations are chosen so that in both wings there are non-local ends, there is some vertex in~$\TF$ different to $X$ that contains a non-local end of~$G$, a contradiction.

Since $\Gamma$ fixes no vertex set of finite diameter, for every $\SF$-separator $S$ and every $r\in\nat$ there is an $\SF$-separator $S'$ with $d_X(S,S')=d(S,S')\ge r$.
Let us say that a component $C$ of $X-B_r(S)$ has the property $(*)$ if
\begin{itemize}
\item[$(*)$] the $\SF$-separators in~$C$ have unbounded distance to $S$.
\end{itemize}
In a first step we show that for any $r>0$ there is a component $C$ of $X-B_r(S)$ with property $(*)$.
So let us assume that there is an $r>0$ such that in each component of $X-B_r(S)$ all $\SF$-separators have bounded distance to~$S$.
Let $S'$ be an $\SF$-separator with $d(S,S')\ge 2r$.
Then $X-B_r(S')$ contains a component $C$ with the property $(*)$ with respect to $S'$ instead of $S$, a contradiction to $S'=S^\alpha$ for some $\alpha\in\Gamma$.
Thus for every $r>0$ there is a component $C$ of $X-B_r(S)$ with property $(*)$.

If on the other hand there is an $r>0$ such that two components $C_1,C_2$ of $X-B_r(S)$ have the property $(*)$, we construct a metric ray in $X$ and thereby show that $X$ have to contain a non-local end.
Let $S_0$ be an $\SF$-separator. Assuming that we have already chosen $\SF$-separators $S_j$ and components $C_j$ of $X-B_r(S_j)$ with $C_j\sub C_{j-1}$ for $j<i$, let $S_i$ be an $\SF$-separator in $C_{i-1}$ with $d(S_i,X-C_{i-1})>r$.
Then there are at least two components of $X-B_r(S_i)$ with $(*)$.
One of those has to lie completely in $C_{i-1}$.
Let $C_i$ be that component.
Fix some vertex $x_i\in S_i$ and let $R_i$ be a path from $x_{i-1}$ to $x_i$.
Then there is a ray $R$ in the union of all the $R_i$.
This ray has to be a metric ray as there are only finitely many vertices on the $R_i$ that have distance smaller than $nr$ for all $n\in\nat$.
Thus $X$ contains some non-local end.

Let us finally suppose that for all $r>0$ there is precisely one component $C_r$ of $X-B_r(S)$ with $(*)$.
Then $C_{r+1}\sub C_r$ for all $r$.
Let $S_i$ be some $\SF$-separator with $d(S,S_i)>i$, and let $x_i$ be some vertex of $S_i$ and $R_i$ some path from $x_i$ to $x_{i+1}$.
Then there is a ray $R$ in the union of all the paths $R_i$.
Again $R$ has to be a metric ray and thus $X$ contains a non-local end of $G$.

So in all cases we either got directly a contradiction or some non-local end in $X$ which also leads to a contradiction as indicated before.
Thus the lemma is proved.
\end{proof}

By replacing the cut system we used for Lemma~\ref{TFHasEnd} by a cut system such that each separation separates local ends we obtain the following corollary.

\begin{Cor}\label{LocEnds=>TreeNoFiniteDiam}
If $G$ is a connected graph with infinitely many local ends such that the automorphism group of~$G$ acts transitively on the ends of~$G$, then either $\Gamma$ fixes a vertex set of finite diameter or any structure tree of~$G$ and of a basic cut system $\SF$ such that each $\SF$-separation separates ends of~$G$ has a ray.\qed
\end{Cor}

A direct consequence of the Corollaries~\ref{LocEndTrans=>NoEndInTree} and \ref{LocEnds=>TreeNoFiniteDiam} is the following theorem.

\begin{Tm}\label{MTm10Local}
Let $G$ be a connected graph with infinitely many ends such that its automorphism group acts transitively on the ends of the graph.
If all ends are local ends, then there is a vertex set of finite diameter that is fixed by $\Aut(G)$.\qed
\end{Tm}

Motivated by the fact that for any graph $G$ with the assumptions on its non-local ends as in this section we have that its structure tree is either a semi-regular tree or a subdivided semi-regular tree, we show in Section~\ref{UniquenessSection} that the semi-regular tree is uniquely determined up to subdivision for each such graph.

\section{Metric ends of end-transitive graphs}

In this section we will show that for every connected graph with infinitely many non-local ends such that no vertex set of finite diameter is fixed by its automorphism group it is equivalent that its automorphism group acts transitively on the non-local ends or on the metric ends of the graph.
Furthermore if the automorphism group of such a graph $G$ is transitive on the non-local ends or on the metric ends, then all non-local ends of~$G$ are thin global ends of~$G$.

Throughout this section let $G$ be a connected graph with infinitely many non-local ends such that its automorphism group $\Gamma$ acts transitively on the non-local ends of~$G$ and such that no vertex set of finite diameter is fixed by~$\Gamma$.
Furthermore let $\SF$ be a basic cut system such that each $\SF$-separation separates non-local ends, and let $\TF$ be the structure tree of~$G$ and $\SF$.

\begin{Lem}\label{Non-LocEndsAreGlobal}
Any thin global end of~$G$ corresponds to an end of~$\TF$ and vice versa.
In particular all non-local ends are global ends.
\end{Lem}

\begin{proof}
By Lemma~\ref{TFHasEnd} the structure tree $\TF$ has infinitely many ends.
We will show that there is a sequence of separations $(A_i,B_i)\in\SF$ such that $A_i\sub A_{i+1}$ and $A_i\cap B_{i+1}=\es$ for all $i\in\nat$.
Suppose that this is not the case.
If any two distinct $\SF$-separators are not disjoint, then the set of all those vertices that lie in any $\SF$-separator is a vertex set of finite diameter, its diameter is bounded by $2\cdot\diam(S)$ for any $\SF$-separator $S$.
Thus we may assume that there are two disjoint $\SF$-separators $S_1,S_2$.
Let $S_3$ be another $\SF$-separator such that $d_\TF(S_1,S_2)=d_\TF(S_2,S_3)$ and $d(S_1,S_3)=2\cdot d_\TF(S_1,S_2)$.
By a similar argument to the one of Lemma~\ref{TranslExists} there is an automorphism $g\in\Gamma$ that acts on $\TF$ like a translation with $S_1^g=S_2$ or $S_1^g=S_3$.
Thus the sequence $(S_1^{g^i})_{i\in\nat}$ is a sequence of pairwise disjoint $\SF$-separators such that each element of that sequence separates its predecessor from its successor.
\end{proof}

We can reformulate the statement of Lemma~\ref{Non-LocEndsAreGlobal} for the following corollary.

\begin{Cor}
Let $G$ be a connected graph with infinitely many non-local ends such that $\Gamma:=\Aut(G)$ acts transitively on the non-local ends of~$G$ and fixes no vertex set of finite diameter.
An end of~$G$ is dominated if and only if it is a local end.\qed
\end{Cor}

\begin{Tm}\label{non-local=metric}
Let $G$ be a connected graph with infinitely many non-local ends and let $\Gamma:=\Aut(G)$.
$\Gamma$ fixes no vertex set of finite diameter.
Then $\Gamma$ acts transitively on the non-local ends of~$G$ if and only if $\Gamma$ acts transitively on the metric ends of~$G$.
\end{Tm}

\begin{proof}
Let $\SF$ be a minimal nested cut system such that the structure tree $\TF$ of~$G$ and $\SF$ is basic.
Every global end of~$G$ must be a metric end since by the transitivity of~$\Gamma$ on the $\SF$-separators the ray of~$\TF$ it corresponds to has to define precisely one metric end.
By Lemma \ref{Non-LocEndsAreGlobal} $\Gamma$ is transitively on the metric ends of~$G$.
On the other hand for every metric end there is a unique non-local end it corresponds to.
Since $\Gamma$ acts transitively on the metric ends, $\Gamma$ also has to be transitive on the non-local ends, as for every non-local end of~$G$ there is at least one metric end corresponding to it.
\end{proof}

The following lemma can be found in~\cite[Corollary~2.5]{PW}.

\begin{Lem}\label{PW2.5}\label{GeodDoubleRays}
For every connected graph $G$ with a separation $(A,B)$ of~$G$ such that $A\cap B$ is finite and with an automorphism $\alpha\in\Aut(G)$ with $A^\alpha\sub A\sm B$ there is some power of $\alpha$ that fixes a geodetic double ray with one end in~$A$ and one end in~$B$.\qed
\end{Lem}

\section{Uniqueness of the structure tree}\label{UniquenessSection}

Our aim is to show that the structure of the tree $\TF$ is essentially independent of the choice of $\SF$.
But Example~\ref{NonUnique} shows that in general it is not unique.
The graph of the example has two different structure trees one of which is the subdivision of the other tree.
But in Theorem~\ref{UniqueTree} we show that this is always the only ambiguity that could occur.

\begin{Exam}\label{NonUnique}
Let $T$ be a subdivision of a semi-regular tree $T'$ that is not regular.
We suppose that $VT'\sub VT$.
Let $A\cup B= VT'$ be be the natural bipartition of~$T'$ and let $C= VT\sm VT'$.
Then all the sets $A,B,C$ are $\Aut(T)$-invariant.
Let $\AF=\{\{a\}\mid a\in A\}$ and let $\BF$ and $\CF$ be the corresponding sets for the sets $B$ and $C$.
Let $\SF_\AF, \SF_\BF, \SF_\CF$ be $\Aut(T)$-invariant cut systems such that the corresponding sets of separators are $\AF$, $\BF$, $\CF$, respectively.
Then the structure trees $\TF_\AF$ and $\TF_\BF$ are isomorphic to~$T'$ and the structure tree $\TF_\CF$ is isomorphic to~$T$.
Thus not all structure trees are isomorphic.
\end{Exam}

\begin{Tm}\label{UniqueTree}
Let $G$ be a connected graph with infinitely many non-local ends such that the automorphism group $\Gamma$ of~$G$ acts transitively on the non-local ends of~$G$ and fixes no vertex set of finite diameter.
Then the structure trees for any two basic cut systems $\SF_1,\SF_2$ such that each $\SF$-separation separates non-local ends are either the same or one is the subdivision of the other.
\end{Tm}

\begin{proof}
We will prove the theorem by a series of claims.

\begin{Claim}
It is sufficient to prove the theorem for each two cut systems such that their union is a nested cut system.
\end{Claim}

\begin{proof}
Let $\SF_1,\SF_2$ be two distinct basic cut systems such that their union is not a nested cut system.
By Lemma~\ref{thirdCutSystem} there is a nested cut system $\SF$ which is nested with $\SF_2$ and such that $m_{\SF_1}(A,B)<m_{\SF_1}(A',B')$ for all $(A,B)\in\SF, (A',B')\in\SF_2$.
By induction on the value $m_{\SF_1}(A,B)$ the structure trees for $\SF$ and $\SF_1$ as well as the structure trees for $\SF$ and $\SF_2$ are essentially the same and by Lemma~\ref{nearlySemi-regular} we also know that the claim holds for the structure trees of~$\SF_1$ and of~$\SF_2$.
\end{proof}

So let $\SF:=\SF_1\cup \SF_2$ be a nested cut system and let $\TF, \TF_1,\TF_2$ be the structure trees of~$G$ and $\SF,\SF_1,\SF_2$, respectively.

\begin{Claim}\label{2ndClaim}
For each two $\SF_1$-separators with distance $2$ in $\TF_1$ there are at most two $\SF_2$-separators separating them.
\end{Claim}

\begin{proof}
Suppose that this is not the case.
Let $A,A'$ be two $\SF_1$-separators with distance $2$ in $\TF_1$ and let $B_1,B_2,B_3$ be three $\SF_2$-separators such that each of them separates $A$ and $A'$.
By Lemma~\ref{TranslExists} there is a translation $\alpha\in\Gamma$ such that one of $B_1,B_2,B_3$ is mapped by~$\alpha$ onto another one of those three separators.
Furthermore $\alpha$ fixes the $\SF_1$-block $X$ between $A$ and~$A'$.
Thus there is an end of~$\TF_2$, name both ends defined by $\alpha$, and hence a corresponding global end in~$G$ by Lemma~\ref{Non-LocEndsAreGlobal} that lies in $X$.
This is a contradiction to the same lemma, as all non-local ends of~$G$ corresponds to ends of the structure trees $\TF_1$ and $\TF_2$.
\end{proof}

So there are at most two $\SF_2$-separators $B_1$ and $B_2$ separating $A$ and $A'$.

\begin{Claim}\label{3rdClaim}
If there are two $\SF_2$-separators $B_1$ and $B_2$ separating $A$ and $A'$, then there are two orbits on the $\SF_2$-blocks of~$G$ in one of which all $\SF_2$-blocks contain $\SF_1$-separators and in one of which no $\SF_2$-block contains any $\SF_1$-separator.
\end{Claim}

\begin{proof}
Let us suppose that this is not the case.
It is not possible that there are several distinct $\Gamma$-orbits on the $\SF_2$-blocks that contain $\SF_1$-separators as $\Gamma$ acts transitively on the $\SF_1$-separators.
Furthermore there cannot be distinct $\Gamma$-orbits on the $\SF_2$-blocks that do not contain any $\SF_1$-separator by the transitivity of~$\Gamma$ on the $\SF_1$-separators and by Claim~\ref{2ndClaim}.

Let $Y_1B_1Y_2B_2$ be a path in~$\TF_2$.
Then there is an $\alpha\in\Gamma$ such that $Y_1^\alpha=Y_2$.
If $\alpha$ does not act like a translation on~$\TF_2$, then $B_1^\alpha= B_1$ and $B_2^{\alpha\inv}\sub Y_1$ and thus there are the three $\SF_2$-separators $B_2^{\alpha\inv},B_1,B_2$ separating two $\SF_1$-separators with distance $2$ in~$\TF_1$ or there is again a non-local end in~$X$.
This contradicts Claim~\ref{2ndClaim} or Lemma~\ref{Non-LocEndsAreGlobal} and thus $\alpha$ is a translation.
Hence there is a unique double ray $R$ in~$\TF_2$ invariant under $\alpha$.
Let $\omega_1,\omega_2$ be the ends that are defined by~$R$.
By Lemma~\ref{Non-LocEndsAreGlobal} there are two non-local ends of~$G$ corresponding to $\omega_1$ and $\omega_2$ and thus there is $\beta\in\Gamma$ with $\omega_1^\beta=\omega_2$.
Then there are two infinite subrays $R_1$ and $R_2$ of~$R$ such that $R_1^\beta=R_2$.
$\beta$ has to map $\SF_2$-separators onto $\SF_2$-separators and $\SF_2$-blocks onto $\SF_2$-blocks.
So let $R_1=x_0x_1\ldots$ and $R_2=y_0y_1\ldots$ such that $x_0$ and $y_0$ are $\SF_2$-blocks and $x_i^\beta=y_i$.
We show that the double ray $R$ has an orientation in~$\TF$ and thus the ends defined by~$R$ cannot be mapped onto each other by some $\gamma\in\Gamma$:
For each $x_i$ with odd~$i$ there is another $\SF_2$-separator in~$x_{i-1}$ but not in~$x_{i+1}$ that is separated from $x_i$ by no $\SF_1$-separator.
Conversely for each $y_i$ with odd~$i$ there is another $\SF_2$-separator in~$y_{i+1}$ but not in~$y_{i-1}$ that is separated from $y_i$ by no $\SF_1$-separator.
Thus the ends $\omega_1$ and $\omega_2$ do not lie in the same $\Gamma$-orbit.
This proves that no $\alpha$ with $Y_1^\alpha=Y_2$ exists.
\end{proof}

We separate the remaining part of the proof into three cases:
In the first one there are two $\Gamma$-orbits on the $\SF_i$-blocks for $i=1,2$, in the second case there is just one $\Gamma$-orbit on the $\SF_1$-blocks but two on the $\SF_2$-blocks, and in the third case there is for each $i=1,2$ just one $\Gamma$-orbit on the $\SF_i$-blocks.
In each case we show the conclusion of the theorem which we denote by $(*)$.

Let $\GF_i$, $i=1,2$, be that set of all $\SF$-blocks such that each element only contains $\SF_i$-separators, and let $\GF_3$ be the set of all other $\SF$-blocks contain both $\SF_1$- and $\SF_2$-separators.

\begin{Claim}
If there are two $\Gamma$-orbits on the $\SF_1$-blocks and on the $\SF_2$-blocks, then $(*)$ holds.
\end{Claim}

\begin{proof}
All $\SF$-separators must have degree~$2$ in $\TF$ and there are at least three $\Gamma$-orbits on the $\SF$-blocks.
Hence we know that $\GF_3\ne\es$.

The elements of $\GF_3$ must have degree $2$ in~$\TF$ as otherwise there would be three $\SF_i$-separators ($i=1$ or $2$) between two $\SF_j$-separators of distance $2$ in~$\TF_j$, $i\neq j$, a contradiction to Claim~\ref{2ndClaim}.
Let $X_1,X_2$ be $\SF_1$-blocks of distinct $\Gamma$-orbits, let $Y_1,Y_2$ be distinct $\SF_2$-blocks of distinct $\Gamma$-orbits, and let $Z_1\in\GF_1$, $Z_2\in\GF_2$.
Then there is w.l.o.g.\ $d_{\TF_1}(X_1)=d_\TF(Z_1)=d_{\TF_2}(Y_1)$ and $d_{\TF_1}(X_2)=d_\TF(Z_2)=d_{\TF_2}(Y_2)$ and thus $(*)$ holds.
\end{proof}

\begin{Claim}
If $\Gamma$ acts transitively on the $\SF_1$-blocks but if there are two $\Gamma$-orbits on the $\SF_2$-blocks, then $(*)$ holds.
\end{Claim}

\begin{proof}
In this case there is $\GF_1=\es$ and $\GF_2,\GF_3\ne\es$.
All elements of~$\GF_2$ must have degree $2$ in~$\TF$ and thus for every $\SF_1$-separator $S$ there is $d_{\TF_1}(S)=d_{\TF_2}(X_1)$ for one (and hence every) $\SF_2$-block $X_1$ containing an $\SF_1$-separator and for every $\SF_1$-block $Y$ there is $d_{\TF_1}(Y)=d_{\TF_2}(X_2)$ for one (and hence every) $\SF_2$-block containing no $\SF_1$-separator.
Thus the claim follows.
\end{proof}

\begin{Claim}
If $\Gamma$ acts transitively on both the $\SF_1$- and the $\SF_2$-blocks, then $(*)$ holds.
\end{Claim}

\begin{proof}
In this case there is $\GF_1=\GF_2=\es$ and for all $\SF_j$-separators $S_j$ and all $\SF_j$-blocks $X_j$ the equality $d_{\TF_j}(S_j)=d_{\TF_i}(X_i)$ with $i\neq j$ holds by Claim~\ref{3rdClaim}.
\end{proof}

This was the last one of the three cases by Lemma~\ref{nearlySemi-regular} and thus the theorem is proved.
\end{proof}

\section{End-transitive graphs}\label{M1Section}

Throughout this section let $G$ be a connected graph with infinitely many non-local ends on which $\Gamma:=\Aut(G)$ acts transitively.
Furthermore $\Gamma$ fixes no vertex set of~$G$ of finite diameter.
Let $\SF$ be a basic cut system and let $\TF$ be the structure tree of~$G$ and $\SF$.

\begin{Lem}\label{DistanceOfSepSmall}
The distance between any two $\SF$-separators in a common $\SF$-block is bounded.
\end{Lem}

\begin{proof}
We have to show that there is a constant $m<\infty$ such that for all two $\SF$-separators $S_1,S_2$ in a common $\SF$-block $X$ there is $d(S_1,S_2)\leq m$.

By Lemma~\ref{TranslExists} there is a translation in $G$ such that on the double ray $Q$ defined by that translation the distance between any two $\SF$-separators with distance $2$ in~$\TF$ has at most two distinct values.
Let us suppose that no such $m$ as conjectured exists.
Then there is a ray $R$ in~$\TF$ such that there is a sequence $\seq{S}{i}{\nat}$ of separators on~$R$ with $d(S_i, S_i')>i$ where $S_i'$ is that $\SF$-separator on~$R$ following on~$S_i$.
But the two described ends, one defined by $R$ and the other one defined by some translation, cannot be mapped onto each other.
This is a contradiction and thus the lemma is proved.
\end{proof}

\begin{Lem}\label{lowerBound}
There is $M<\infty$ such that each vertex of~$G$ lies in at most $M$ $\SF$-blocks on each ray in~$\TF$.
\end{Lem}

\begin{proof}
Suppose the claim does not hold.
We construct a sequence $\seq{P}{i}{\nat}$ of finite paths with $P_i\sub P_{i+1}$ such that there is a subpath of length $i$ of $P_i$ such that the intersection of the blocks and separators on that subpath is not empty.
Then for every finite path $P_i$ of length at least $i$ with endvertex $X$ in~$\TF$ there is an infinite component $C$ of $\TF-P_i$ that is adjacent to~$X$ and in which there is a ray containing $2i+2$ separator vertices with a non-trivial intersection as $\Gamma$ acts transitively on the separator vertices of~$\TF$.
The elements of that non-trivial intersection might intersect trivially with any separator of the path $P_i$.
We may extend $P_i$ and get a finite path $P_{i+1}$ such that there is a sequence of at least $i+1$ separator vertices containing a vertex $y_{i+1}\in VG$.
By recursion we get a ray $R=\bigcup _{i\in\nat} P_i$ in $\TF$ such that for each $i\in\nat$ there is a sequence of length $i$ of separator vertices on $R$ that intersect non-trivially.

By Lemma~\ref{Non-LocEndsAreGlobal} there is a global end $\omega$ of~$G$ defined by~$R$.
By Lemma~\ref{TranslExists} there is an automorphism of~$G$ that acts on~$\TF$ like a translation.
The unique double ray of~$\TF$ fixed by that automorphism has an infinite subray $R'$.
Let $\omega'$ be the end of~$G$ defined by~$R'$.
Since $\Gamma$ acts transitively on the non-local ends, there is some $g\in\Gamma$ with $\omega^g=\omega'$.
But then we may assume that $R^g$ has only finitely many vertices distinct from $R'$ and thus we may also assume that $R^g\sub R'$.

If we finally show that on~$R'$ any vertex of~$G$ lies in only $m$ separator vertices of~$R'$ for a constant $m$, then we get a contradiction and this would prove the lemma.
But if there is a sequence $\sequ{x}$ such that $x_i$ lies in at least $i$ separators on~$R'$, then each separator must contain infinitely many vertices, as the translation maps any separator at most $2$ separators apart, in contradiction to the definition of the $\SF$-separators.
\end{proof}

\begin{Lem}\label{upperBoundSemiRegTree}
Let $G$ be quasi-isometric to a semi-regular tree $T$.
Then the set of all $\SF$-blocks and all $\SF$-slices has bounded diameter.
\end{Lem}

\begin{proof}
Suppose the lemma is false.
Let us first assume that the set of all $\SF$-blocks has bounded diameter.
By Lemma~\ref{nearlySemi-regular} there is an $\SF$-block $X$ that has no finite diameter.
Then with Lemma~\ref{DistanceOfSepSmall} there is a sequence $\seq{x}{i}{\nat}$ in~$X$ such that $\min\set{d(x_i,S)\mid S\text{ $\SF$-separator}}\ge i$ for all $i\in\nat$.
Let $x$ be a vertex in $S$ for an $\SF$-separator $S\sub X$ and let $t\in VT$ be the vertex with $x^\varphi=t$ for the quasi-isometry $\varphi:G\to T$.
Let $\seq{t}{i}{\nat}$ be a sequence in $VT$ with $x_i^\varphi=t_i$.
This sequence has an infinite subsequence $\seq{t}{i}{I}$ with $I\sub\nat$ of pairwise distinct elements.
By Lemma~\ref{DistanceOfSepSmall} there is $d(x,y)<M$ for some constant $M<\infty$ and all $y\in S'$ with $S'$ $\SF$-separator in $X$.
Thus there is an $r\in\nat$ such that $B_r(t)$ contains the images of all the vertices $y\in S'$ for all $\SF$-separators $S'\sub X$.
Then there is a component $C$ of $T-B_r(t)$ that contains at least one vertex $t_i$.
Since $T$ is a semi-regular tree, $C$ contains a ray $R$.
Let $R'$ be a set of vertices in~$X$ such that there is an $r'\in\nat$ with $R\sub B_{r'}((R')^\varphi)$.
As $G$ is quasi-isometric to~$T$, there is at least one non-local end of~$G$ defined by~$R'$ and this has to lie in~$X$, a contradiction to Lemma~\ref{Non-LocEndsAreGlobal}.

So let us assume that the $\SF$-slices have unbounded diameter.
Let $S$ be an $\SF$-separator and let $\YF$ be the set of all those slices $Y$ such that $Y$ is a component of $G-S$.
Then $X:=\bigcup \YF$ contains no metric ray as otherwise there would be a non-local end in $X$ but all such ends corresponds to ends of $\TF$ by Lemma~\ref{Non-LocEndsAreGlobal}.
Thus there is no ray of~$T$ whose preimages lie in~$X$ by a similar argument as in the first case.
Hence $X$ has only finite diameter.
\end{proof}

\begin{Lem}\label{upperBoundGeodRays}
Assume that there is some $\delta\ge 0$ such that any vertex of~$G$ lies at distance at most $\delta$ to the union of all geodetic double rays.
Then the set of all $\SF$-blocks and all $\SF$-slices has bounded diameter.
\end{Lem}

\begin{proof}
By Lemmas~\ref{TranslExists} and \ref{PW2.5} the intersection of each $\SF$-separator with all geodetic double rays is not empty.
As all $\SF$-separators have the same diameter and the distance between any two $\SF$-separators in the same $\SF$-block is bounded by Lemma~\ref{DistanceOfSepSmall}, there is an upper bound on the distance from each vertex in any $\SF$-block to any $\SF$-separator.
Additionally the distance from each vertex of any $\SF$-slice to any $\SF$-separator is bounded by $\delta+s$ where $m$ denotes the diameter of any $\SF$-separator.
\end{proof}

All the previously proved lemmas enable us to prove Theorem~\ref{MTm10}.

\begin{Tm}\label{MTm10}
Let $G$ be a connected graph with infinitely many non-local ends such that $\Gamma:=\Aut(G)$ acts transitively on the non-local ends and such that $\Gamma$ fixes no vertex set of finite diameter.
Then the following assertions holds:
\begin{enumerate}[(a)]
\item For every $x\in VG$ there is an $r\in\nat$ such that $G(x,r)$ covers all geodetic double rays of $G$.
\item For every $x\in VG$ there is an $R\in\nat$ such that the graph $G-G(x,R)$ contains no metric ray of~$G$.
\item The following statements are equivalent:
\begin{enumerate}[(i)]
\item There is an $r\in\nat$ such that every vertex is $r$-close to some geodetic double ray;
\item $\Gamma$ acts metrically almost transitively on $G$;
\item there is a $\Gamma$-congruence $\pi$ such that a vertex set of finite diameter meets every congruence class of $\pi$;
\item $G$ is quasi-isometric to any basic structure tree of~$G$;
\item $G$ is quasi-isometric to a semi-regular tree with minimum degree $2$.
\end{enumerate}
\item All of the properties of part~(c) hold for the subgraph $G(x,R)$ of property~(b).
\end{enumerate}
\end{Tm}

\begin{proof}
Let $\SF$ be a basic cut system and let $\TF$ be the structure tree of~$G$ and~$\SF$.
Let $m$ be the constant of Lemma~\ref{DistanceOfSepSmall} and let $s$ be the diameter of any $\SF$-separator.
Let $B$ be the $(2m+2s)$-ball around an arbitrary $\SF$-separator $S$.
Then $B$ covers the intersection of each geodetic double rays in $G$ with every $\SF$-block $X$ with $S\sub X$ and $B$ covers also the intersection of each geodetic double rays with every $\SF$-slices $Y$ such that $Y$ is a component of $G-S$.
So if we set $d$ as the minimum over all $d(x,S)$ for $\SF$-separators $S$, then the statement of part~(a) holds for $r:=2m+2s+d$.

If we set $R:=2m+2s+d$ (where $d$ denotes the same value as before), then any $R$-ball around $x$ covers all $\SF$-separators and since no metric ray lies in any $\SF$-block by Lemma~\ref{Non-LocEndsAreGlobal} and by definition no metric ray lies in any $\SF$-slice. Thus we just proved part (b).

The assertion (d) is an immediate consequence of (b) and (c) as the graph $G(x,R)$ of (b) is by construction a metrically almost transitive graph.
So we just have to prove the equivalences of~(c).
The equivalence of (i) and (ii) follows with Lemma~\ref{GeodDoubleRays} and Lemma~\ref{DistanceOfSepSmall} from the definition of metrically almost transitive graphs.
The equivalence of~(ii) and (iii) is just the definition of metrically almost transitive graphs.

By the Lemmas~\ref{lowerBound} and \ref{upperBoundGeodRays} the condition (i) of this theorem immediately implies (iv).
So let us assume that (iv) holds, that is $G$ is quasi-isometric to~$\TF$.
Then there is some constant $k$ such that every vertex of~$G$ lies at distance at most~$k$ to some $\SF$-separator.
By Lemma~\ref{GeodDoubleRays} each $\SF$-separator meets some geodetic double ray of~$G$ and thus if the $\SF$-separators have diameter $s$, then every vertex of~$G$ lies at most $s+k$ apart from any geodetic double ray.

As the last part of the proof of Theorem~\ref{MTm10} we show the equivalence of~(iv) and (v).
If $G$ is quasi-isometric to a semi-regular tree, then the Lemmas~\ref{lowerBound} and \ref{upperBoundSemiRegTree} imply that $G$ is quasi-isometric to any basic structure tree of~$G$.
So let us assume that $G$ is quasi-isometric to a basic structure tree $\TF$.
By Lemma~\ref{nearlySemi-regular} the structure tree $\TF$ is quasi-isometric to a semi-regular tree and thus $G$ is also quasi-isometric to a semi-regular tree.
\end{proof}

If we just require $\Gamma$ to act transitively on the ends of~$G$ instead of the non-local ends of~$G$, we get Theorem~\ref{MTm10Short} as a corollary of Theorem~\ref{MTm10}.

We finish this section with an observation: Let $G$ be a graph as in Theorem~\ref{MTm10Short} and let $H$ be any rayless graph. If we add to each vertex $y\in VG$ a copy $H_y$ of~$H$ with an additional edge, then for the constructed graph $G'$ there is a component of $G-G(x,R)$ (where $x$ and $R$ are as in the theorem) that is isometric to $H$.
So any rayless graph can occur as a component of $G-G(x,R)$.

\section{On the stabilizer of an end}

For this section let $G$ be a connected graph with infinitely many non-local ends such that the stabilizer of some end $\omega$ acts transitively on the vertices of~$G$.
Let $\Gamma:=\Aut(G)$.
Let $\SF$ be a $\Gamma_\omega$-invariant, not necessarily $\Gamma$-invariant, basic cut system such that every $\SF$-separation separates ends of~$G$ and let $\TF$ be the structure tree of~$G$ and~$\SF$.
We remark that the end $\omega$ has to be a global end since otherwise it would be dominated by some vertex $x$ and thus by every vertex as $\Gamma_\omega$ acts transitively on the vertices of~$G$.

To prove Theorem~\ref{mainTmPart2OfM} and Theorem~\ref{mainTmPart2OfMb} we prove some lemmas.

\begin{Lem}\label{endToEnd}
The end $\omega$ of~$G$ corresponds to an end of $\TF$ and not to a vertex of~$\TF$.
\end{Lem}

\begin{proof}
Let us suppose that $\omega$ corresponds to a block vertex $X$ of~$\TF$.
Let $x\in A\cap B$ for some separation $(A,B)\in\SF$ with $A\cap B\sub X$.
Let $y$ be some vertex of~$G$ not in $A\cap B$ which is separated by $A\cap B$ from $\omega$.
Since $\Gamma_\omega$ acts transitively on $VG$, there is some $\alpha\in\Gamma_\omega$ with $y^\alpha=x$.
Then $(A\cap B)^\alpha$ is a separator separating $x=y^\alpha$ from $\omega$, a contradiction to the choice of $A\cap B$ and $x$.
Hence $\omega$ corresponds to an end of~$\TF$.
\end{proof}

\begin{Lem}\label{DiamBound}
The diameter of all $\SF$-blocks is globally bounded.
Furthermore each vertex of an $\SF$-block $X$ has distance at most $1$ to that separator that separates $X$ from $\omega$.
\end{Lem}

\begin{proof}
Let $X$ be some $\SF$-block and let $(A,B)$ be a separation in~$\SF$ with $A\cap B\sub X$ such that $A\cap B$ separates $X$ from $\omega$.
We will prove that any vertex $x\in X\setminus (A\cap B)$ has distance $1$ to $A\cap B$.

Suppose $d(x,A\cap B)=2$.
Let $y$ be a vertex in $X$ with $d(y,A\cap B)=1$.
There is some $\alpha\in\Gamma_\omega$ with $y^\alpha=x$.
Then $(A\cap B)^{\alpha\inv}$ is a separator with distance $1$ to $x$ such that $x$ and $\omega$ are separated by $(A\cap B)^{\alpha\inv}$.
By the choice of $x$ and $(A,B)$ this is a contradiction.
Thus we know that $\diam(X)\leq\diam(A\cap B)+2$.
Since all $\SF$-separations have the same order, the claim follows.
\end{proof}

\begin{Lem}\label{UniqueDetOfBlocks}
For every $\SF$-block $X$ there is a vertex $x$ such that $x$ determines the block $X$ (with respect to the end $\omega$).
Additionally for every vertex $x$ there is a unique $\SF$-separator $S$ separating $x$ from $\omega$ with $x\notin S$.
\end{Lem}

\begin{proof}
These are direct consequences of Section~6 from \cite{DK}.
\end{proof}

\begin{Lem}\label{BlocksIso}
$\Gamma_{\omega}$ acts transitively on the $\SF$-blocks.
\end{Lem}

\begin{proof}
Let $X$ and $Y$ be $\SF$-blocks.
By Lemma~\ref{UniqueDetOfBlocks} there are vertices $x\in X$ and $y\in Y$ such that $x$ and $\omega$ determine $X$ and $y$ and $\omega$ determine $Y$.
Let $\alpha\in\Gamma_\omega$ with $x^\alpha=y$.
Then $X^\alpha=Y$.
\end{proof}

Let us define the tree order on the vertices of $\TF$ with respect to the end $\omega_\TF$ of~$\TF$ that corresponds to $\omega$:
$x\leq y$ if and only if $x$ separates $y$ from $\omega_\TF$.
Let $X,Y$ be the blocks or separators corresponding to $x$, $y$, respectively.
Then $x\leq y$ if and only if $X$ separates $Y$ from $\omega$ in~$G$.
As for rooted trees let $\lfloor x\rfloor$ denote all vertices $y$ in $\TF$ with $y\geq x$.

\begin{Lem}\label{EndOfTreeGlobalEnd}
Any global end of~$G$ corresponds to an end of~$\TF$ and vice versa.
\end{Lem}

\begin{proof}
By Lemma~\ref{DiamBound} any global end $\omega'$ of $G$ corresponds to an end of $\TF$.
Suppose that there is some end of $\TF$ whose corresponding end $\omega'$ in $G$ is not a global end.
Then the end $\omega'$ must be dominated by some vertex.
Let $S_1,S_2,\ldots$ be a sequence of separators such that $S_i$ separates $S_{i-1}$ from $S_{i+1}$, $S_1$ separates $S_2$ from $\omega$, and every $S_i$ separates $\omega'$ from $\omega$.
Then there is no infinite pairwise disjoint subsequence of the $S_i$ and hence there is some vertex $x$ contained in infinitely many $S_i$.
We may assume that $x\in S_i$ for all $i\in\nat$ and that $S_1$ is an $\SF$-separator that contains $x$ and that is minimal in the tree order with this property.
Let $S:=S_1$.
Since $\Gamma_\omega$ acts transitively on $\SF$, the stabilizer of~$S$ in $\Gamma_\omega$ acts transitively on the vertices in $\lfloor S\rfloor$ with equal distance to $S$.
There is a ray in $\lfloor S\rfloor$ such that every vertex contains $x$.
Thus every vertex in $\lfloor S\rfloor$ contains an element of $S$.
Hence for every ray in $\lfloor S\rfloor$ that starts in~$S$ the intersection of all vertices on that ray is not empty and thus contains an element of~$S$.
As $\Gamma_\omega$ acts transitively on $VG$, every vertex $y$ that is separated from $\omega$ by~$S$ lies in the intersection of the vertices of a ray of~$\TF$.
Hence all $\SF$-separators $S'\in\lfloor S\rfloor$ with $d(S,S')\ge 2i$ have to contain a vertex from the finite set
$$\bigcup\{\overline{S}\in\lfloor S\rfloor\mid d(S,\overline{S})<2i\}\sm \bigcup\{\overline{S}\in\lfloor S\rfloor\mid d(S,\overline{S})<2(i-1)\}.$$
But then all $\SF$-separators $S'$ with $d(S,S')\ge 2\abs{S}$ have to contain more than $\abs{S}$ vertices which is impossible.
\end{proof}

Let us now prove Theorem~\ref{mainTmPart2OfM}.

\begin{proof}[Proof of Theorem~\ref{mainTmPart2OfM}]
Let $\SF$ be a basic cut system and let $\TF$ be the structure tree of~$G$ and $\SF$.
By the Lemmas~\ref{endToEnd} and \ref{EndOfTreeGlobalEnd} all non-local ends are global ends and also $|S|$-thin ends for all $\SF$-separators $S$.
Thus it suffices to prove the denseness condition for $G$.
We will first prove this condition for $\TF$.
Let us identify the unique end in $\TF$ that corresponds to the end $\omega$ of~$G$ with $\omega$.
Let $\widetilde{\omega},\widehat{\omega}$ be any two distinct ends of $\TF$ that are both different to $\omega$.
We have to show that in any open neighbourhood around $\widetilde{\omega}$ there is some $\Gamma_\omega$-image of $\widehat{\omega}$.
It suffices to show this for any neighbourhood of the form $\lfloor t\rfloor$ for some $t\in V\TF$.

Let $x$ be a vertex on the unique double ray in $\TF$ between $\widetilde{\omega}$ and $\widehat{\omega}$ such that $x$ is minimal in the tree order.
As $\Gamma_\omega$ acts on $V\TF$ with precisely two orbits, there is an automorphism $g\in\Gamma_\omega$ such that either $x^g=t$ or $d(x^g,t)=1$ and $x^g\in\lfloor t\rfloor$.
Since $g$ fixes $\omega$, the end $\widehat{\omega}^g$ has to lie in $\lfloor t\rfloor$.
As any open neighbourhood around $\omega$ contains an end of~$\TF$ that is different to $\omega$, any neighbourhood of~$\omega$ contains some $\widehat{\omega}^g$ with $g\in\Gamma_\omega$.
Thus $\widehat{\omega}^{\Gamma_\omega}$ is dense in the set of non-local ends.

Let $\widetilde{\omega}$ be a local end of~$G$.
As for every $\SF$-separators $S$ the group $\Gamma_\omega$ acts transitively on those $\SF$-separators that have distance $2$ in~$\TF$ to~$S$ and that are separated from $\omega$ by~$S$, each ray $R$ in $\widetilde{\omega}$ meets at least infinitely many of those $\SF$-separators.
There is a block vertex $X\in V\TF$ such that $\widetilde{\omega}$ corresponds to $X$.
Then $X$ must have infinitely many neighbours in~$\TF$ and also in~$G$.
As $\widehat{\omega}^{\Gamma_\omega}$ is dense in $\Omega\TF$, for each finite vertex set $S$ of~$G$ there is an end $\widehat{\omega}^g$ with $g\in\Gamma$ that is not separated from $X$ by $S$ and thus that is also not separated from $\widehat{\omega}$.

Now let $\widehat{\omega}$ be a local end of~$G$.
By Lemma~\ref{EndOfTreeGlobalEnd} there is an $\SF$-block $X$ such that $\widehat{\omega}$ corresponds to~$X$.
Let first $\widetilde{\omega}$ be a global end of~$G$.
Then for each $\SF$-separator $S$ there is an $X^g$ with $g\in\Gamma_\omega$ in the same component of $G-S$ in which $\widetilde{\omega}$ lies.
So let $\widetilde{\omega}$ be a local end of~$G$.
Analog to~$\widehat{\omega}$, there is an $\SF$-block $Y$ such that $\widetilde{\omega}$ corresponds to the $\TF$-vertex $Y$.
Since $\Gamma_\omega$ acts transitively on the $\SF$-blocks, every finite vertex set can separate $Y$ only from finitely many $X^g$ with $g\in\Gamma_\omega$ and thus $\widetilde{\omega}$ lies in the closure of~$\widehat{\omega}$.
\end{proof}

The following theorem, Theorem~\ref{mainTmPart2OfMbNearly}, immediately implies Theorem~\ref{mainTmPart2OfMb}.

\begin{Tm}\label{mainTmPart2OfMbNearly}
Let $G$ be a connected graph with infinitely many ends such that for some end $\omega$ of~$G$ the stabilizer of~$\omega$ in the automorphism group $\Gamma$ of~$G$ acts transitively on the vertices of $G$.
Then $G$ is quasi-isometric to any basic structure tree of~$G$ which is a semi-regular tree with minimum degree $2$.
\end{Tm}

\begin{proof}
Let $\SF$ be a basic cut system and let $\TF$ be the structure tree of~$G$ and $\SF$.
By Lemma~\ref{BlocksIso} the tree $\TF$ is a semi-regular tree and by Lemma~\ref{DiamBound} all vertices of~$\TF$ have bounded diameter in~$G$ and thus the claim holds.
\end{proof}

A direct consequence is the following corollary.

\begin{Cor}
Let $G$ be a graph with infinitely many ends such that the stabilizer of an end acts transitively on the vertices.
Then any non-local end of~$G$ is a thin global end of~$G$.\qed
\end{Cor}

\end{document}